\documentclass[12pt,a4paper]{report}
\let\openright=\clearpage
\usepackage[a-2u]{pdfx}
\usepackage[utf8]{inputenc}
\usepackage{lmodern}
\usepackage{amsmath}      
\usepackage{amsfonts}
\usepackage{amsthm}
\usepackage{bbding}
\usepackage{bm}
\usepackage{graphicx}
\usepackage{fancyvrb}
\usepackage{natbib}
\usepackage[nottoc]{tocbibind}
\usepackage{dcolumn}
\usepackage{booktabs}
\usepackage{paralist}
\usepackage{xcolor}
\usepackage{enumitem}
\usepackage{amssymb}
\usepackage{verbatim}

\setlist[enumerate]{label={\rm(\roman*)}}
\newtheorem{theorem}{Theorem}
\newtheorem{corollary}[theorem]{Corollary}
\newtheorem{proposition}[theorem]{Proposition}
\newtheorem{lemma}[theorem]{Lemma}
\theoremstyle{definition}
\newtheorem{remark}[theorem]{Remark}
\numberwithin{theorem}{chapter}

\def\ManuscriptTitle{Optimality of embeddings in Orlicz Spaces}

\def\ManuscriptAuthor{Tomáš Beránek}

\def\Abstract{%
Working with function spaces in various branches of mathematical analysis introduces optimality problems, where the question of choosing a function space both accessible and expressive becomes a nontrivial exercise. A good middle ground is provided by Orlicz spaces, parameterized by a single Young function and thus accessible, yet expansive. In this work, we study optimality problems on Sobolev embeddings in the so-called Maz'ya classes of Euclidean domains which are defined through their isoperimetric behavior. In particular, we prove the non-existence of optimal Orlicz spaces in certain Orlicz--Sobolev embeddings in a limiting (critical) state whose pivotal special case is the celebrated embedding of Brezis and Wainger for John domains.
}

\def\Keywords{%
    integral operator, Orlicz space, optimality, Sobolev embedding.
}

\hypersetup{unicode}
\hypersetup{breaklinks=true}

\makeatletter
\def\@makechapterhead#1{
  {\parindent \z@ \raggedright \normalfont
   \Huge\bfseries \thechapter. #1
   \par\nobreak
   \vskip 20\p@
}}
\def\@makeschapterhead#1{
  {\parindent \z@ \raggedright \normalfont
   \Huge\bfseries #1
   \par\nobreak
   \vskip 20\p@
}}
\makeatother

\theoremstyle{plain}

\theoremstyle{plain}

\theoremstyle{remark}

\newtheorem*{example}{Example}

\DefineVerbatimEnvironment{code}{Verbatim}{fontsize=\small, frame=single}

\begin{document}

{\bfseries\MakeUppercase{\ManuscriptTitle}}
\\

\vbox to 0.5\vsize{
\setlength\parindent{0mm}
\setlength\parskip{5mm}

Author:
\ManuscriptAuthor

Email address: tomas.beranek@matfyz.cuni.cz

Institution:
Department of Mathematical Analysis, 
Faculty of Mathematics and Physics, Charles University

Address of Institution:
Sokolovská 83, 186 75 Praha 8 

Abstract:
\Abstract


Acknowledgement: 
The research was partially supported by grant no. 23-04720S of the Czech Science Foundation. 

\vss}

\chapter*{Introduction}
\addcontentsline{toc}{chapter}{Introduction}

When tasked with creating a mathematical model of a given scenario, one approach is to interpret the input data and output solutions as measurable functions, which belong to some function spaces, and the assignment of data to a solution as an operator between these spaces. We then obtain the abstract model
\begin{equation}\label{E:map}
    T\colon X \rightarrow Y,
\end{equation}
where $X$ and $Y$ are function spaces containing the data and solutions, respectively, and $T$ is a linear operator mapping the data to its solution. For example, in economics, the problem of increasing one's capital in the stock market can be modelled as the relation \eqref{E:map}, where $X$ is the behaviour of the market, $Y$ are possible decisions of a trader (purchases or sales), and $T$ maps the current state of the market to the best possible decision. Numerous examples of such modelling in physics, biology and other fields follow naturally.

One of the most important questions when working with such models is, what can we say about the quality of the solutions based on the quality of the input data. This is the motivation behind using different function spaces, as we group together functions with the same quality into one space. At this point, it is natural to ask whether we may find an optimal space for a given problem. There are two forms of such a task: either we are given data and search for the smallest group of possible solutions (= given space $X$ and finding the smallest space $Y$), or we are given conditions on the solutions and search for the largest amount of data whose solutions satisfy such demands (= given space $Y$ and finding the largest space $X$).

When solving optimality problems, it is important to remember that we also have to consider the properties of the chosen classes of function spaces in terms of accessibility and expressivity. For instance, Lebesgue $L^p$ spaces are a very well understood and accessible class of spaces (as they are described by a single parameter), though the spaces may be too sparse to provide accurate enough information, and are as such not expressive enough for certain needs. On the other side of the spectrum sit rearrangement-invariant spaces, where an optimal space virtually always exists, however, the optimal space is described implicitly, and is as such practically impossible to work with. A good middle ground is provided by Orlicz spaces, a class of function spaces described by Young functions. As such, these spaces provide both accessibility and expressivity.

Optimality problems are not a new discipline - the earlier results go back into the 2nd half of the 20th century, but it was not until the turn of the millennium that they saw a boom in interest. As such, the field is now supported by a vast amount of literature, which includes the works \cite{HeMoTr70}, \citet{Ad75}, \citet{Ta76}, \citet{BrWa80}, \citet{Ci96}, \citet{CuRi02}, \citet{Vy07}, \citet{FoMo14}, \citet{ClSo16}, \citet{NeOp20}, and \citet{Ho21}.

We focus on the optimal form of Sobolev embeddings within rearrangement-invariant spaces, and within Orlicz spaces, on Maz'ya classes of Euclidean domains. Using the general theory we introduce, we prove the nonexistence of certain optimal Orlicz spaces in Orlicz--Sobolev embeddings, namely that there is no largest domain Orlicz space in the embedding
\begin{equation*}
    W^mL^A(\Omega) \hookrightarrow L^{\infty, q, -1+m(1-\alpha) - \frac{1}{q}}(\Omega)
\end{equation*}
for appropriate values of the parameters involved.

The motivation for studying embeddings into this particular space stems from the fact that it is the optimal (smallest) rearrangement-invariant space which renders the embedding true for every $\Omega\in \mathcal J_{\alpha}$ in the very important case when the corresponding domain Orlicz space $L^A$ is the critical (limiting) Lebesgue space $L^{\frac{1}{m(1-\alpha)}}$. This was first observed by~\cite{BrWa80} in connection with the special case $\alpha=\frac{1}{n'}$, that is, $L^{\frac{1}{m(1-\alpha)}}=L^{\frac nm}$, where $n$ is the dimension of the underlying ambient Euclidean space and $m$ is the order of differentiation. There are various reasons for establishing results involving such Lorentz-type refinements on the target side, perhaps the most notable one being the fact that when such embeddings are considered, then no loss of information occurs under their iterations (for more details, see~\citet{CiPiSl15} and the references therein).

We shall now describe our principal results in detail. We first establish a general formula for the fundamental function of an operator-induced space based on the isoperimetric behavior of the underlying domain under certain mild assumptions. Next, we consider a specific situation concerning Maz'ya domains. We then establish a theorem which enables us to transfer the information of nonexistence of an optimal Orlicz domain partner space from a Marcinkiewicz space to any space on the same fundamental level. We finally apply this general scheme to a limiting Sobolev embedding. We thereby obtain a wide variety of results applicable to any Maz'ya domain having isoperimetric profile $t^{\alpha}$ with $\alpha\in[\frac{1}{n'},1)$ whose particular case for Lipschitz domains was obtained in the earlier work~\cite{MuPiTa22}.  

The text is structured as follows. In Chapter 1, we present definitions and basic knowledge about the relevant function spaces and the isoperimetric function. In Chapter 2, we collect background results necessary for proofs of our main results, namely the principal alternative and reduction principle for Sobolev embeddings, and the forms of optimal r.i.~domain and target spaces. Lastly, in Chapter 3, we present the main results of this work. 

\chapter{Preliminaries}
\section{Function spaces} \label{S: Function spaces}

In this section, we recall some definitions and basic facts from the theory of various function spaces. For further details, the standard reference is \citet{BS}.

Let $n \in \mathbb{N}$, $n \geq 2$. In this work, $|E|$ denotes the $n$-dimensional Lebesgue measure of $E$ for $E \subset \mathbb{R}^n$ measurable. We use the convention that $\frac{1}{\infty} = 0$ and $0 \cdot \infty = 0$.

Let $\Omega \subset \mathbb{R}^n$ be a bounded open set. We assume, without loss of generality, that $\lvert \Omega \rvert = 1$, and define
\begin{equation*}
    \mathcal{M}(\Omega) = \{ f\colon\Omega \rightarrow [ -\infty,\infty ];~f~ \mbox{is~Lebesgue-measurable~in~}(0,\infty)\},
\end{equation*}
\begin{equation*}
    \mathcal{M}_+(\Omega) = \{ f \in \mathcal{M}(\Omega)\colon f \geq 0 \mbox{ a.e.}\},
\end{equation*}
and
\begin{equation*}
    \mathcal{M}_0(\Omega) = \{ f \in \mathcal{M}(\Omega)\colon f~\mathrm{is~finite~a.e.~in~}\Omega \}. 
\end{equation*}

The \emph{distribution~function}~$f_*\colon(0, \infty) \rightarrow [0, \infty]$ of a function $f \in \mathcal{M}(\Omega)$ is defined as
\begin{equation*}
    f_*(s) = \lvert \{ x \in \Omega \colon \lvert f(x) \rvert > s \} \rvert \quad \mathrm{for}~ s \in (0,\infty),
\end{equation*}
and the 
\emph{non-increasing~rearrangement}~$f^*\colon [0,1] \rightarrow [0, \infty]$ of a function $f \in \mathcal{M}(\Omega)$ is defined as
\begin{equation*}
    f^*(t) = \inf{\{ s \in (0, \infty)\colon f_*(s) \leq t \} } \quad \mathrm{for}~ t \in [0,1].
\end{equation*}
The operation $f \rightarrow f^*$ is monotone in the sense that for $f_1,f_2 \in \mathcal{M}(\Omega)$,
\begin{equation*}
    \lvert f_1 \rvert \leq \lvert f_2 \rvert \quad \mathrm{a.e.~in~}\Omega \implies f_1^* \leq f_2^* \quad \mathrm{in}~ [0,1].
\end{equation*}
The \emph{elementary~maximal~function}~$f^{**}\colon(0,1] \rightarrow [0,\infty]$ of a function $f \in \mathcal{M}(\Omega)$ is defined as
\begin{equation*}
    f^{**}(t) = \frac{1}{t} \int_0^t f^*(s) \, \mathrm{d}s \quad \mathrm{for}~ t \in (0,1].
\end{equation*}
The operation $f \rightarrow f^{**}$ is subadditive in the sense that for $f_1,f_2 \in \mathcal{M^{+}}(\Omega)$,
\begin{equation*}
    (f_1+f_2)^{**} \leq f_1^{**} + f_2^{**}.
\end{equation*}

The \emph{Hardy--Littlewood~inequality} is a classical property of function rearrangements, which asserts that, for $f_1, f_2 \in \mathcal{M}(\Omega)$,
\begin{equation}\label{H-L}
    \int_{\Omega} \lvert f_1(x)f_2(x)\rvert \, \mathrm{d}x \leq \int_0^1 f_1^*(t)f_2^*(t) \, \mathrm{d}t.
\end{equation}
A specialization of the inequality states that for every $f \in \mathcal{M}(\Omega)$ and for every $E \subset \Omega$ measurable,
\begin{equation*}
    \int_{E} \lvert f(x)\rvert \, \mathrm{d}x \leq \int_0^{\lvert E \rvert} f^*(t) \, \mathrm{d}t.
\end{equation*}

Next, we define the rearrangement-invariant norm. We say that a functional
\begin{equation*}
    \| \cdot \|_{X(0,1)} \colon \mathcal{M}_+(0,1) \rightarrow [0, \infty] 
\end{equation*}
is a \emph{function~norm}, if for all $f, g$ and $\{ f_j \}_{j \in \mathbb{N}}$ in $\mathcal{M}_+(0,1)$, and every $\lambda \geq 0$, the following properties hold:
\begin{enumerate}[label=(P\arabic*),itemindent=*]
    \item\label{(P1)}$\| f \|_{X(0,1)} = 0 \iff f = 0$ ~a.e.,
        
        $\| \lambda f \|_{X(0,1)} = \lambda \| f \|_{X(0,1)}$,

        $\| f+g \|_{X(0,1)} \leq \| f \|_{X(0,1)} + \| g \|_{X(0,1)} $;
    \item\label{(P2)}$f \leq g$ a.e.~$\implies \| f \|_{X(0,1)} \leq \| g \|_{X(0,1)}$;
    \item\label{(P3)}$f_j \nearrow f$ a.e.~$\implies \| f_j \|_{X(0,1)} \nearrow \| f \|_{X(0,1)}$;
    \item\label{(P4)}$\| 1 \|_{X(0,1)} < \infty$;
    \item\label{(P5)}$\int_0^1 f(x) \, \mathrm{d}x \leq c \cdot \| f \|_{X(0,1)}$ for some constant $c$ independent of $f$.
\end{enumerate}
If, in addition, the property
\begin{enumerate}[label=(P6)]
    \item\label{(P6)}$\| f \|_{X(0,1)} = \| g \|_{X(0,1)} $ if $ f^* = g^*$ 
\end{enumerate}
holds, we call the functional $\| \cdot \|_{X(0,1)}$ a \emph{rearrangement-invariant~function~norm}.
For any such rearrangement-invariant function norm $\| \cdot \|_{X(0,1)}$, we define the functional $\| \cdot \|_{X'(0,1)}$ as
\begin{equation*}
    \| f \|_{X'(0,1)} = \sup \big\{ \int_0^1 f(t)g(t) \, \mathrm{d}t \colon g \in \mathcal{M}_+(0,1), \| g \|_{X(0,1)} \leq 1 \big\}
\end{equation*}
for $f \in \mathcal{M}_+(0,1)$. The functional $\| \cdot \|_{X'(0,1)}$ is then also a rearrangement-invariant function norm, see \cite[Chapter~1,~Theorem~2.2]{BS}, and it is called the \emph{associate~function~norm} of $\| \cdot \|_{X(0,1)}$ and, by \cite[Chapter~1,~Theorem~2.7]{BS} it holds that $\| \cdot \|_{X''(0,1)} = \| \cdot \|_{X(0,1)}$. We say that $\| \cdot \|_{X(0,1)}$ is a \emph{rearrangement-invariant function quasinorm}, if it satisfies the conditions \ref{(P2)}, \ref{(P3)}, \ref{(P4)} and \ref{(P6)}, and \ref{(Q1)}, a weaker version of \ref{(P1)}, where
\begin{enumerate}[label=(Q\arabic*),itemindent=*]
    \item\label{(Q1)}$\| f \|_{X(0,1)} = 0 \iff f = 0$ ~a.e.,
        
        $\| \lambda f \|_{X(0,1)} = \lambda \| f \|_{X(0,1)}$,

        There exists$  ~c \in (0, \infty) ~ \mbox{such~that} ~ \| f+g \|_{X(0,1)} \leq c \cdot ( \| f \|_{X(0,1)} + \| g \|_{X(0,1)} )$,
\end{enumerate}
for all $f,g \in \mathcal{M}_+(0,1)$ and every $\lambda \geq 0$.

Given a rearrangement-invariant function norm $\| \cdot \|_{X(0,1)}$, we define the functional $\| \cdot \|_{X(\Omega)}$ as
\begin{equation*}
    \| f \|_{X(\Omega)} = \| f^* \|_{X(0,1)} \quad \mathrm{for}~f \in \mathcal{M}(\Omega),
\end{equation*}
and we call the set
\begin{equation*}
    X(\Omega) = \{ f \in \mathcal{M}(\Omega)\colon \| f \|_{X(\Omega)} < \infty \} 
\end{equation*}
a \emph{rearrangement-invariant~space}. Furthermore, the space $X(0,1)$ is called the \emph{representation~space} of $X(\Omega)$, and we define the \emph{associate~space}~$X'(\Omega)$ of $X(\Omega)$ as
\begin{equation*}
    X'(\Omega) = \{ f \in \mathcal{M}(\Omega)\colon \| f \|_{X'(\Omega)} < \infty \}.
\end{equation*}
Then, the \emph{H\"older~inequality}
\begin{equation*}
    \int_{\Omega} |f(x)g(x)| \, \mathrm{d}x \leq \| f \|_{X(\Omega)} \| g \|_{X'(\Omega)}
\end{equation*}
holds for every $f \in X(\Omega)$ and $g \in X'(\Omega)$. 

For any rearrangement-invariant spaces $X(\Omega)$ and $Y(\Omega)$, the continuous embedding of $X(\Omega)$ into $Y(\Omega)$ is denoted by $X(\Omega) \hookrightarrow Y(\Omega)$ and means that there exists a constant $c > 0$ such that for any $f \in X(\Omega)$, it holds that $f \in Y(\Omega)$ and $\|f\|_{Y(\Omega)} \leq c \cdot \|f\|_{X(\Omega)}$. By \cite[Chapter~1,~Proposition~2.10]{BS} it holds that
\begin{equation*}
    X(\Omega) \hookrightarrow Y(\Omega) \iff Y'(\Omega) \hookrightarrow X'(\Omega),    
\end{equation*}
and by \cite[Chapter~1,~Theorem~1.8]{BS}, it holds that
\begin{equation*}
    X(\Omega) \subset Y(\Omega) \implies X(\Omega) \hookrightarrow Y(\Omega).
\end{equation*}

Note that the functional $\| \cdot \|_{X(\Omega)}$ may also be defined if its corresponding functional $\| \cdot \|_{X(0,1)}$ is only a rearrangement-invariant quasinorm. However, some of the properties listed here for the case where $\| \cdot \|_{X(0,1)}$ is a rearrangement-invariant norm then do not necessarily hold.

Occasionally, when no confusion can arise, we will, for simplicity's sake, omit the underlying domain in the notation, more precisely, we will write $X$ in place of $X(\Omega)$ or $X(0,1)$, etc.

For given rearrangement-invariant spaces $X$ and $Y$, we denote the boundedness of an operator $T$ from $X$ to $Y$ by $T\colon X \rightarrow Y$.

Let $\lambda > 0$. For any $f \in \mathcal{M}(0,1)$, the \emph{dilation~operator}~$E_{\lambda}$ is defined as
\begin{equation*}
    E_{\lambda}f(t) = \left\{
	\begin{array}{ll}
		f(\tfrac{t}{\lambda})  & \mathrm{if}~ 0 < t \leq \lambda \\
	    0  & \mathrm{if}~ \lambda < t \leq 1.
	\end{array}
    \right.
\end{equation*}
Such an operator is bounded on any rearrangement-invariant space $X(0,1)$, with its norm smaller than or equal to max\{$1, \tfrac{1}{\lambda}$\}.

We introduce, for $\gamma \in (0,1)$, the operator
\begin{equation}\label{E:sup-operator}
    T_\gamma \, g(t) = t^{-\gamma} \sup_{s \in [t,1]} s^\gamma g^*(s), \quad g \in \mathcal{M}(0,1),~t \in (0,1).
\end{equation}

Given any $f_1, f_2 \in \mathcal{M}_+(0,1)$ such that
\begin{equation*}
    \int_0^t f_1(s) \, \mathrm{d}s \leq \int_0^t f_2(s) \, \mathrm{d}s \quad \mathrm{for~every~} t \in (0,1),
\end{equation*}
by \emph{Hardy's~lemma} the inequality
\begin{equation*}
    \int_0^1 f_1(s)h(s) \, \mathrm{d}s \leq \int_0^1 f_2(s)h(s) \, \mathrm{d}s
\end{equation*}
holds for every non-increasing function $h\colon (0,1) \rightarrow [0,\infty]$. Consequently, the \emph{Hardy--Littlewood--P\'olya~principle}, which asserts that if $g_1, g_2 \in \mathcal{M}(\Omega)$ satisfy
\begin{equation*}
    \int_0^t g^*_1(s) \, \mathrm{d}s \leq \int_0^t g_2^*(s) \, \mathrm{d}s \quad \mathrm{for~every~} t \in (0,1),
\end{equation*}
then
\begin{equation*}
    \|g_1\|_{X(\Omega)} \leq \|g_2\|_{X(\Omega)}
\end{equation*}
holds for every rearrangement-invariant space $X(\Omega)$.

The \emph{fundamental~function}~$\varphi_X\colon [0,1] \rightarrow [0,1]$ of a rearrangement-invariant space $X(\Omega)$ is defined as
\begin{equation*}
    \varphi_X(t) = \|\chi_{E}\|_{X(\Omega)} \quad \mathrm{for}~ t \in [0,1],
\end{equation*}
where $E\subset \Omega$ is measurable and such that $|E|=t$. Thanks to the rearrangement invariance of $\|\cdot\|_{X(\Omega)}$, the function $\varphi_X$ is well defined. By \cite[Proposition~1.1]{Sh72}, the fundamental function is locally absolutely continuous on $(0,1]$. We define the \emph{fundamental~level} as the collection of all rearrangement-invariant spaces, which share the same fundamental function.

We say that a function $\varphi\colon [0,\infty) \rightarrow [0,\infty)$ is \emph{quasiconcave}, if $\varphi(t) = 0$ if and only if $t = 0$, it is positive and non-decreasing on $(0,\infty)$, and the function $\frac{t}{\varphi(t)}\colon (0,\infty) \rightarrow (0, \infty)$ is non-decreasing. Recall that by \cite[Chapter~2~Corollary~5.3]{BS} for any rearrangement-invariant space $X(\Omega)$, its fundamental function $\varphi_X$ is quasiconcave. Furthermore, by \cite[Chapter~2~Proposition~5.10]{BS} it holds that for any quasiconcave function $\varphi\colon [0,\infty) \rightarrow [0,\infty)$ there exists a concave function $\overline{\varphi}\colon [0,\infty) \rightarrow [0,\infty)$ such that for every $t \in [0, \infty),$ the inequality $\tfrac{1}{2}\overline{\varphi}(t) \leq \varphi(t) \leq \overline{\varphi}(t)$ holds.

Let $\varphi\colon [0,\infty) \rightarrow [0,\infty)$ be a quasiconcave function and let $\overline{\varphi}\colon [0,\infty) \rightarrow [0,\infty)$ be a concave function such that $\tfrac{1}{2}\overline{\varphi}(t) \leq \varphi(t) \leq \overline{\varphi}(t)$ for every $t \in [0, \infty)$. We then define the functionals
\begin{equation*}
     \|f\|_{\Lambda_\varphi} = \int_0^\infty f^*(t)\,\mathrm{d}\overline{\varphi}(t), \quad f \in \mathcal{M}(\Omega),
\end{equation*}
and
\begin{equation*}
    \|f\|_{M_\varphi} = \sup_{t \in (0, \infty)} \varphi(t)f^{**}(t), \quad f \in \mathcal{M}(\Omega).
\end{equation*}
By \cite[Chapter~2,~Theorem~5.13]{BS}, these functionals are rearrangement-invariant function norms, and as such, we define the corresponding rearrangement-invariant spaces $\Lambda_\varphi$, and $M_\varphi$. Both of these spaces have the same fundamental function equivalent to $\varphi$. Furthermore, the space $\Lambda_\varphi$ is the smallest rearrangement-invariant space with the fundamental function $\overline{\varphi}$, while $M_\varphi$ is the largest rearrangement-invariant space with the fundamental function $\varphi$.

Given a rearrangement-invariant space $X$, we define the corresponding \emph{Lorentz space} $\Lambda(X) = \Lambda_{\varphi_X}$ and \emph{Marcinkiewicz space} $M(X) = M_{\varphi_X}$. Let us now recall the Lorentz--Marcinkiewicz sandwich

\begin{equation} \label{E:Lorentz-Marcinkiewicz-Sandwich}
	\Lambda(X) \hookrightarrow X \hookrightarrow M(X).
\end{equation}

We recall the embeddings into Lebesgue spaces for any rearrangement-invariant space $X(\Omega)$
\begin{equation*}
    L^\infty(\Omega) \hookrightarrow X(\Omega) \hookrightarrow L^1(\Omega).
\end{equation*}
Let $1 \leq p,q \leq \infty $. The functionals $\| \cdot \|_{L^{p,q}(0,1)}$ and $\| \cdot \|_{L^{(p,q)}(0,1)}$ are respectively defined as
\begin{equation*}
\begin{split}
    & \| f \|_{L^{p,q}(0,1)} = \|t^{\frac{1}{p}-\frac{1}{q}}f^*(t)\|_{L^q(0,1)} \quad \mathrm{and} \\
    & \| f \|_{L^{(p,q)}(0,1)} = \|t^{\frac{1}{p}-\frac{1}{q}}f^{**}(t)\|_{L^q(0,1)}
\end{split}
\end{equation*}
for $f \in \mathcal{M}_+(0,1)$. Let us recall that if $1 < p \leq \infty$,
\begin{equation*}
    L^{p,q}(\Omega) = L^{(p,q)}(\Omega),
\end{equation*}
and if one of the conditions
\begin{enumerate}[label=(L\arabic*),itemindent=*]
    \item $1 < p < \infty, ~1 \leq q \leq \infty$,
    \item $p = q = 1$,
    \item $p = q = \infty$,
\end{enumerate}
is met, then $\| \cdot \|_{L^{p,q}(0,1)}$ is equivalent to a rearrangement-invariant function norm. Then, the corresponding rearrangement-invariant function space $L^{p,q}(\Omega)$ is called a \emph{Lorentz~space}.

Let $p \in [1,\infty]$. Then $L^p(\Omega) = L^{p,p}(\Omega)$, and
\begin{equation*}
    1 \leq r \leq s \leq \infty \implies L^{p,r}(\Omega) \hookrightarrow L^{p,s}(\Omega),
\end{equation*}
where, if $p < \infty$, equality of the spaces is attained if and only if $r = s$. Note that by equality of two rearrangement-invariant spaces $X$ and $Y$, we mean that $X$ and $Y$ coincide in set-theoretical sense, and, moreover, that their norms are equivalent in the sense that there exists a constant $c > 0$ such that
\begin{equation*}
    c^{-1} \cdot \| f \|_X \leq \|f\|_Y \leq c \cdot \|f\|_X
\end{equation*}
for every $f \in X$.

Let $1 \leq p,q \leq \infty,~ \alpha \in \mathbb{R}$. The functionals $\| \cdot \|_{L^{p,q;\alpha}(0,1)}$ and $\| \cdot \|_{L^{(p,q;\alpha)}(0,1)}$ are defined as
\begin{equation*}
\begin{split}
    & \| f \|_{L^{p,q;\alpha}(0,1)} = \|t^{\frac{1}{p}-\frac{1}{q}} \, \log ^{\alpha}\big( \tfrac{2}{t}\big) f^*(t)\|_{L^q(0,1)} \quad \mathrm{and} \\
    & \| f \|_{L^{(p,q;\alpha)}(0,1)} = \|t^{\frac{1}{p}-\frac{1}{q}} \, \log ^{\alpha}\big( \tfrac{2}{t}\big) f^{**}(t)\|_{L^q(0,1)}
\end{split}
\end{equation*}
for $f \in \mathcal{M}_+(0,1)$. If one of the conditions
\begin{enumerate}[label=(Z\arabic*),itemindent=*]
    \item\label{(Z1)}$1 < p < \infty$, $1 \leq q\leq \infty$, $\alpha \in \mathbb{R}$,
    \item\label{(Z2)}$p = 1,~ q = 1,~ \alpha \geq 0$,
    \item\label{(Z3)}$p = \infty,~ q = \infty,~ \alpha \leq 0$,
    \item\label{(Z4)}$p = \infty, ~1 \leq q < \infty, ~\alpha + \tfrac{1}{q} < 0$,
\end{enumerate}
is met, then $\| \cdot \|_{L^{p,q;\alpha}(0,1)}$ is equivalent to a rearrangement-invariant function norm. Then the corresponding rearrangement-invariant function space $L^{p,q; \alpha}(\Omega)$ is called a \emph{Lorentz--Zygmund}~space.

We say that $A\colon [0,\infty) \rightarrow [0,\infty]$ is a \emph{Young~function}, if it is a convex non-constant left-continuous function such that $A(0) = 0$. Let us recall that any such function may be written in the integral form
\begin{equation*}
    A(t) = \int_0^t a(s) \mathrm{d}s \quad \mathrm{for}~ t \geq 0,
\end{equation*}
where $a\colon[0,\infty) \rightarrow [0,\infty]$ is a non-decreasing, left-continuous function, which is not identically $0$ or $\infty$.

The \emph{Luxemburg~function~norm} is defined as
\begin{equation*}
    \|f\|_{L^A(0,1)} = \inf\left\{\lambda > 0\colon \int_0^1 A\left( \frac{f(s)}{\lambda} \right) \mathrm{d}s \leq 1 \right\} \quad \mathrm{for}~ f \in \mathcal{M_+}(0,1).
\end{equation*}
By \cite[Chapter~4, Theorem~8.9]{BS}, the Luxemburg function norm $\| \cdot \|_{L^A(0,1)}$ is a rearrangement-invariant function norm. The \emph{Orlicz~space} $L^A(\Omega)$ is defined as the rearrangement-invariant space associated with the Luxemburg function norm. Then, for some $p \in [1,\infty)$ and $A(t) = t^p$, $L^A(\Omega) = L^p(\Omega)$; and for $B(t) = \infty \cdot \chi_{(1,\infty)}(t)$, $L^B(\Omega) = L^{\infty}(\Omega)$.

Let $A$ and $B$ be Young functions. We say that A dominates B near infinity if there exist constants $c>0$ and $t_0>0$ such that
\begin{equation*}
    B(t) \leq A(ct) \quad \mathrm{for}~t\geq t_0.
\end{equation*}
We say that A and B are equivalent near infinity if they dominate each other near infinity. Furthermore, it holds that
\begin{equation*}
    L^A(\Omega) \hookrightarrow L^B(\Omega) \iff A ~\mathrm{dominates}~ B ~ \mathrm{near~infinity.}
\end{equation*}
We denote certain Orlicz spaces without explicitly defining the corresponding Young functions. The Orlicz space associated with a Young function equivalent near infinity to $t^p\log ^\alpha t$, where $p>1$ and $\alpha \in \mathbb{R}$, or $p=1$ and $\alpha \geq 0$, is denoted by  $L^p\log ^\alpha L(\Omega)$, and the Orlicz space associated with a Young function equivalent near infinity to $e^{t^\beta}$, where $\beta > 0$, is denoted by exp$\,L^\beta (\Omega)$. 

In certain cases, the classes of Lorentz-Zygmund and Orlicz spaces overlap. If either $1 < p < \infty$, $\alpha \in \mathbb{R}$, or $p = 1$, $\alpha \geq 0$, then $L^{p,p;\alpha}(\Omega) = L^p\log ^{p\alpha}L(\Omega) $. Additionally, if $\beta > 0$, then 
\begin{equation}\label{Orlicz-exp}
    L^{\infty, \infty;-\beta}(\Omega) = \mathrm{exp}\,L^{\frac{1}{\beta}}(\Omega).
\end{equation}

For certain classes of function spaces, their fundamental functions are known. By \citet{OpPi99}, it holds that
\begin{equation*}
    \varphi_{L^p}(t) = \left\{
    \begin{array}{ll}
        t^{\frac{1}{p}} & \mbox{if~} p < \infty \\
        1 & \mbox{if~} p = \infty,
    \end{array}
    \right. 
\end{equation*}
\begin{equation*}
    \varphi_{L^{p,q}} = \varphi_{L^p} \quad \mbox{for all } p,q,
\end{equation*}
and
\begin{equation*}
    \varphi_{L^{p,q; \alpha}} \approx \left\{
    \begin{array}{ll}
        t^\frac{1}{p} \log ^\alpha \tfrac{2}{t} & \mbox{if~conditions~\ref{(Z1)}~or~\ref{(Z2)}~are~met} \\
        (\log \tfrac{2}{t})^{\alpha + \frac{1}{q}} & \mbox{if~conditions~\ref{(Z3)}~or~\ref{(Z4)}~are~met}.
    \end{array}
    \right. 
\end{equation*}

Furthermore, by \cite[Example 7.9.4 (iv), p. 260]{PKJF}, it holds that
\begin{equation} \label{Orlicz-fundamental_formula}
    \varphi_{L^A}(t) = \frac{1}{A^{-1}(\frac{1}{t})}.    
\end{equation}
Therefore, for every fundamental level of rearrangement-invariant spaces, there exists a unique Orlicz space with the same fundamental function. This \emph{fundamental Orlicz space} is denoted by $L(X)$, where $X$ is a rearrangement-invariant space.

Let $h_1$, $h_2 \in \mathcal{M}^+(\Omega)$. The fact that there exists a positive constant $c$ such that $h_1(x) \leq c \cdot h_2(x)$, or $h_1(x) \geq c \cdot h_2(x)$ for any $x \in \Omega$ is denoted by $h_1 \lesssim h_2$ or $h_1 \gtrsim h_2$, respectively. If both inequalities $h_1 \lesssim h_2$ and $h_2 \gtrsim h_1$ hold, we write $h_1 \approx h_2$.

Let $A$, $B \colon (0, \infty) \rightarrow (0, \infty)$. By $A(t) \approx B(t)$ for $t \gg 1$, we denote that there exist constants $c$, $t_0 > 1$ such that $c^{-1} \cdot A(t) \leq B(t) \leq c \cdot A(t)$ for every $t > t_0$.

We say that $A\colon (0, \infty) \rightarrow (0,\infty)$ satisfies the $\Delta_2$ condition, if it is non-decreasing and the inequality $A(2t) \lesssim A(t)$ holds for every $t \in (0,\infty)$. Then, by \citet[Theorem 4.7.3]{PKJF} it holds that
\begin{equation*}
    f \in L^A(\Omega) \iff \int_\Omega A(f) < \infty \quad \mbox{for~any~} f \in \mathcal{M}_+(\Omega).
\end{equation*}

\section{Isoperimetric functions and Sobolev spaces}

In this section, we shall define some basic terms and recall simple facts concerning the isoperimetric function and Sobolev spaces.

Let $n,~\Omega$ be as in Section \ref{S: Function spaces}. Let $n' = \frac{n}{n-1}.$ We define the \emph{perimeter} $P(E,\Omega)$, of a Lebesgue-measurable set $E \subset \Omega$ as
\begin{equation*}
    P(E,\Omega) = \int_{\Omega \cap \partial^M E} \mathrm{d} \mathcal{H}^{n-1}(x),
\end{equation*}
where $\partial^ME$ denotes the essential boundary of $E$ (for details see \citet{Maz11}. We then define the \emph{isoperimetric function} $I_\Omega\colon [0,1] \rightarrow [0,\infty]$ of $\Omega$ as
\begin{equation*}
    I_\Omega(t) = \left\{
        \begin{array}{ll}
            \inf \{ P(E,\Omega)\colon E \subset \Omega,~ t \leq |E| \leq \tfrac{1}{2} \}  & \mbox{if~} t \in [0, \tfrac{1}{2}], \\
            I_\Omega(1-t) & \mbox{if~} t \in (\tfrac{1}{2},1].
        \end{array}
    \right.
\end{equation*}

Note that $I_\Omega(t)$ is finite for $t \in [0,\tfrac{1}{2})$ (for the detailed proof, see \cite[Chapter~4]{CiPiSl15}). Also, by \cite[Proposition~4.1]{CiPiSl15}, there exists a constant $c>0$ such that
\begin{equation*}
    I_\Omega(t) \leq c \cdot t^\frac{1}{n'} \quad \mbox{near~zero}.
\end{equation*}
Thus, the best possible behavior of the isoperimetric function at 0 is $I_\Omega(t) \approx t^{\frac{1}{n'}}$. What this means is that, essentially, $I_\Omega(t)$ cannot decay more slowly than $t^\frac{1}{n'}$ as $t \rightarrow 0$.

We will say that $\Omega$ has a \emph{Lipschitz boundary}, or simply is a \emph{Lipschitz domain}, if at each point of the boundary of $\Omega$, the boundary is locally the graph of a Lipschitz function.

We shall call $\Omega$ a \emph{John domain}, if there exists $c \in (0,1)$ and $x_0 \in \Omega$ such that for every $x \in \Omega$ there exists a rectifiable curve parametrized by arclength $\xi \colon [0,r] \rightarrow \Omega$, such that $\xi(0) = x$, $\xi(r) = x_0$ and
\begin{equation*}
    \mathrm{dist}(\xi(s),\partial\Omega) \geq c \cdot s \quad \mbox{for}~s \in [0,r].
\end{equation*}
An important link between John domains and the theory of isoperimetric functions is that if $\Omega$ is a John domain, then $I_\Omega(t) \approx t^\frac{1}{n'}$.

Let $\alpha \in [\frac{1}{n'},1]$. We define the \emph{Maz'ya class} of Euclidean domains $G$ as
\begin{equation*}
    \mathcal{J}_\alpha = \big\{ G\colon I_G(t) \geq c \cdot t^\alpha \quad \mbox{~for~some~constant~}c>0 \mbox{~and~} t \in [0,\tfrac{1}{2}] \big\}.
\end{equation*}

We observe that all Lipchitz domains are John domains and the inclusion $\mathcal{J}_\alpha \subseteq \mathcal{J}_\beta$ holds for $\alpha \leq \beta$.

Let $m \in \mathbb{N}$, let $X(\Omega)$ be a rearrangement-invariant function space. The $m$\emph{-th order Sobolev space} $W^mX(\Omega)$ is defined as
\begin{equation*}
    W^mX(\Omega) = \{ u\colon u \mbox{~is~} m \mbox{-times~weakly~differentiable~in~} \Omega,
\end{equation*}
\begin{equation*}
    \quad |\nabla^ku| \in X(\Omega) \mbox{~for~} k = 0,\dots,m \}, 
\end{equation*}
where $\nabla^ku$ denotes the vector of partial derivatives of $u$ of order $k$, and the $m$\emph{-th order Sobolev space} $V^mX(\Omega)$ is defined as
\begin{equation*}
    V^mX(\Omega) = \{ u\colon u \mbox{~is~} m \mbox{-times~weakly~differentiable~in~} \Omega,
\end{equation*}
\begin{equation*}
    |\nabla^mu| \in X(\Omega) \}. \qquad \qquad \quad \quad ~ \,
\end{equation*}
Assume now that
\begin{equation} \label{E: Isoperimetric-Sobolev-equality}
    \int_0^1 \frac{1}{I_\Omega(s)} \, \mathrm{d}s < \infty .
\end{equation}
Then, by \cite[Proposition~4.5]{CiPiSl15}
\begin{equation*}
    W^mX(\Omega) = V^mX(\Omega)    
\end{equation*}
in set-theoretical sense with their norms equivalent. If we only consider a weaker form of \eqref{E: Isoperimetric-Sobolev-equality}, namely that there exists a positive constant $c$ such that
\begin{equation*}
    I_\Omega(t) \geq c \cdot t \quad \mbox{for~} t \in [0, \tfrac{1}{2}],
\end{equation*}
it then holds by \cite[Chapter~4, Corollary 4.3, Proposition 4.4]{CiPiSl15} that $V^mX(\Omega) \hookrightarrow V^mL^1(\Omega)$, $ V^mX(\Omega) \subset V^kL^1(\Omega) $ for every $k = 0,\dots,m-1,$ and furthermore for any $Y(\Omega)$ rearrangement-invariant space, $V^mX(\Omega) \hookrightarrow Y(\Omega)$ if and only if there exists a positive constant $c$ such that $ \|u\|_{Y(\Omega)} \leq c \cdot \| \nabla^mu \|_{X(\Omega)}$ for all $u \in V_\perp^mX(\Omega)$, where
\begin{equation*}
    V_\perp^mX(\Omega) = \big\{ u \in V^mX(\Omega)\colon \int_\Omega \nabla^ku(x) \, \mathrm{d}x = 0 \quad \mathrm{for~} k = 0,\dots,m-1 \big\}. 
\end{equation*}

In the case where $X(\Omega) = L^A(\Omega)$ is an Orlicz space, we define an \emph{Orlicz--Sobolev space} as $W^mL^A(\Omega) = W^mX(\Omega)$.

We say that $X(\Omega)$ is the optimal (largest) rearrangement-invariant domain space in the embedding
\begin{equation} \label{E:prelim-sobolev}
    W^m X(\Omega) \hookrightarrow Y(\Omega)
\end{equation}
if $X(\Omega)$ is a rearrangement-invariant space, embedding \eqref{E:prelim-sobolev} holds, and if \eqref{E:prelim-sobolev} holds with $X(\Omega)$ replaced by a rearrangement-invariant space $Z(\Omega)$, then the embedding $Z(\Omega) \hookrightarrow X(\Omega)$ holds. Similarly, we say that $Y(\Omega)$ is the optimal (smallest) r.i. target space in embedding \eqref{E:prelim-sobolev}, if $Y(\Omega)$ is a rearrangement-invariant space, embedding \eqref{E:prelim-sobolev} holds, and if embedding \eqref{E:prelim-sobolev} holds with $Y(\Omega)$ replaced by a rearrangement-invariant space $V(\Omega)$, then the embedding $Y(\Omega) \hookrightarrow V(\Omega)$ holds.

\chapter{Background results}

Let $n \in \mathbb{N}$, $n \geq 2$, let $\Omega \subset \mathbb{R}^n$ be an open set, and let $X(\Omega)$ be a rearrangement-invariant space. For the sake of brevity, we shall refer to rearrangement-invariant spaces as r.i.~spaces from this point onwards. In this work, we consider Sobolev spaces $W^mX(\Omega)$ together with their norm defined as
\begin{equation*}
    \|u\|_{W^mX(\Omega)} = \sum_{k = 0}^m \|\nabla^ku\|_{X(\Omega)}
\end{equation*}
for $m \in \mathbb{N}$. Furthermore, we consider Sobolev embeddings of the form
\begin{equation} \label{E:sobolev}
	W^m X(\Omega) \hookrightarrow Y(\Omega),
\end{equation}
where $Y(\Omega)$ is an r.i.~space. We restrict ourselves to such sets $\Omega$ which fulfil the property
\begin{equation*} 
    \inf_{t \in (0,1)} \frac{I_\Omega(t)}{t} > 0
\end{equation*}
and classes of such sets. Furthermore, it is known that given such an r.i.~space $Y(\Omega)$, the optimal r.i.~domain space always exists and can be explicitly described. Namely, the combination of a reduction principle from \citet{CiPiSl15} with the characterization of the optimal domain for a Copson integral operator from \citet[Proposition~3.3]{Mi23}, such optimal r.i.~space $X(\Omega)$ obeys
\begin{equation} \label{E:optimal-ri-domain}
	\|u\|_{X( \Omega )} = \sup_{h} \left\Vert \int_{t}^{1} \frac{h(s)}{I_{\Omega}(s)} \, \mathrm{d}s \right \Vert _{Y(0,1)}~~~~ \mathrm{for}~u\in \mathcal{M}(\Omega),
\end{equation}
where the supremum is taken over all $h \in \mathcal{M}_+(0,1)$ such that $h^* = u^*$.

We first examine possible approaches to the reduction of Sobolev embeddings to significantly simpler one-dimensional inequalities for Maz'ya domains. Such problems have already been examined and solved, and as such, for our purposes, it suffices to use the reduction principle stated and proven in \cite[Theorem 6.4]{CiPiSl15}, which follows. For the proof, see the original paper.

\begin{theorem}[reduction principle for Maz'ya domains] \label{T:reduction-principle-mazya}
    Let $n \in \mathbb{N}$, $n \geq 2$, $m \in \mathbb{N}$ and $\alpha \in [\tfrac{1}{n'},1).$ Let $\| \cdot \|_{X(0,1)}$ and $\| \cdot \|_{Y(0,1)}$ be r.i.~function norms. Let $c > 0$ be such that
    \begin{equation} \label{E: reduction-principle-mazya-inequality}
        \left\Vert \int_t^1 f(s)s^{-1+m(1-\alpha)} \, \mathrm{d}s \right\Vert_{Y(0,1)} \leq c \cdot \|f\|_{X(0,1)}
    \end{equation}
    for any $f \in X(0,1)$ nonnegative. Then, the Sobolev embedding \eqref{E:sobolev}
    holds for every $\Omega \in \mathcal{J}_\alpha$. Conversely, if the Sobolev embedding \eqref{E:sobolev} holds for every $\Omega \in \mathcal{J}_\alpha$, then the inequality \eqref{E: reduction-principle-mazya-inequality} holds.
\end{theorem}

Note that we have omitted the case $\alpha = 1$. While our results may be, after some modification, applied to such a case, it remains rather technical and is beyond the scope of this work.

As a consequence of Theorem \ref{T:reduction-principle-mazya}, we can identify the optimal r.i.~target space $Y(\Omega)$ associated with a given domain space $X(\Omega)$ in the Sobolev embedding \eqref{E:sobolev} for any $\Omega \in \mathcal{J}_\alpha$. This is also a known result, for the proof see \cite[Theorem 6.5]{CiPiSl15}.

\begin{theorem}[optimal r.i.~target for Maz'ya domains] \label{T:optimal-ri-target-mazya}
    Let $n,~m,~\alpha$ and $\| \cdot \|_{X(0,1)}$ be as in Theorem \ref{T:reduction-principle-mazya}. Define the functional $\| \cdot \|_{X'_{m,\alpha}(0,1)}$ as
    \begin{equation*}
        \| \cdot \|_{X'_{m,\alpha}(0,1)} = \left\Vert t^{-1+m(1-\alpha)} \int_0^t f^*(s) \, \mathrm{d}s \, \right\Vert_{X'(0,1)}
    \end{equation*}
    where $\| \cdot \|_{X'(0,1)}$ denotes the associate function norm to $\| \cdot \|_{X(0,1)}$. Then, the functional $\| \cdot \|_{X'_{m,\alpha}(0,1)}$ is an r.i.~function norm, whose associate norm $\| \cdot \|_{X_{m,\alpha}(0,1)}$ satisfies
    \begin{equation} \label{E: optimal-ri-target-mazya-embedding}
        W^mX(\Omega) \hookrightarrow X_{m, \alpha}(\Omega)
    \end{equation}
    for every $\Omega \in \mathcal{J}_\alpha$, and for some constant $c$ depending on $\Omega$, $m$, $X(\Omega)$ and $Y(\Omega)$
    \begin{equation} \label{E: optimal-ri-target-mazya-inequality}
        \|u\|_{X_{m,\alpha}(\Omega)} \leq c \cdot \|\nabla^mu\|_{X(\Omega)}
    \end{equation}
    for every $\Omega \in \mathcal{J}_\alpha$ and every $u \in V_\perp^mX(\Omega)$. Furthermore, the function norm $\| \cdot \|_{X_{m,\alpha}(0,1)}$ is optimal in \eqref{E: optimal-ri-target-mazya-embedding} and \eqref{E: optimal-ri-target-mazya-inequality} among all r.i.~norms, as $\Omega$ ranges in $\mathcal{J}_\alpha$.
\end{theorem}

We shall now use Theorem \ref{T:optimal-ri-target-mazya} to show the optimal target r.i.~space for certain critical r.i.~spaces.
\begin{example}
Let $ \alpha \in [\tfrac{1}{n'},1)$. Then, for any $\Omega \in \mathcal{J}_\alpha$ and for any $q \in (1,\infty]$, given the critical spaces $W^mL^{\frac{1}{m(1-\alpha)},q}(\Omega)$, we obtain the embedding
\begin{equation*}
    W^mL^{\frac{1}{m(1-\alpha)},q}(\Omega) \hookrightarrow L^{\infty, q; -1+(1-\alpha)m - \frac{1}{q}}(\Omega),
\end{equation*}
where $L^{\infty, q; -1+(1-\alpha)m - \frac{1}{q}}(\Omega)$ is the smallest (= optimal) space with this property for any $\Omega \in \mathcal{J}_\alpha$. In the case $q=1$, we obtain the embedding
\begin{equation*}
    W^mL^{\frac{1}{m(1-\alpha)},1}(\Omega) \hookrightarrow L^{\infty}(\Omega),
\end{equation*}
where $L^{\infty}(\Omega)$ is the optimal target space. In the case $q = \infty$, we obtain the embedding
\begin{equation*}
    W^mL^{\frac{1}{m(1-\alpha)},\infty}(\Omega) \hookrightarrow \mathrm{exp} \, L^{\frac{1}{1-m(1-\alpha)}}(\Omega),
\end{equation*}
where $\exp L^{\frac{1}{1-m(1-\alpha)}}(\Omega)$ is the optimal target space. In the case $q=\frac{1}{m(1-\alpha)}$, we obtain the embedding
\begin{equation}\label{E:general-bw}
    W^mL^{\frac{1}{m(1-\alpha)}}(\Omega) \hookrightarrow L^{\infty,\frac{1}{m(1-\alpha)};-1}(\Omega),
\end{equation}
where $L^{\infty,\frac{1}{m(1-\alpha)};-1}(\Omega)$ is the optimal target space, and which recovers, as its special case for $\alpha=\frac{1}{n'}$, the result of~\cite{BrWa80}.
\end{example}

The question remains whether we can say anything about the optimality of r.i.~domains, given a target r.i.~space. This problem has also been extensively studied, and as such, for our purposes, it suffices to modify a known result by \cite[Theorem 3.3]{KerPi06}.
\begin{theorem}[optimal r.i.~domain space] \label{T:optimal-ri-domain-mazya}
    Let $m$, $n \in \mathbb{N}$, $\alpha \in [\tfrac{1}{n'},1)$. Let $Y(0,1)$ be an r.i.~space such that
    \begin{equation}\label{E:KPembedding}
        Y(0,1) \hookrightarrow L^{\frac{1}{1-m(1-\alpha)},1}(0,1).
    \end{equation}
    Define the functional $\| \cdot \|_{X(0,1)}$ as
    \begin{equation*}
        \| f \|_{X(0,1)}  =  \sup_{h \sim f} \left\Vert \int_t^1 s^{-1+m(1-\alpha)}h(s) \, \mathrm{d}s \right\Vert_{Y(0,1)} \quad \mbox{for~} f \in \mathcal{M}(0,1),
    \end{equation*}
    where the supremum is extended over all $h \in \mathcal{M}_+(0,1)$ such that $h^* = f^*$, and the set
    \begin{equation*}
        X(0,1) = \big\{ f \in \mathcal{M}(0,1)\colon \| f \|_{X(0,1)} <\infty \big\}.
    \end{equation*}    
    Then, the embedding
    \begin{equation*}
        W^mX(\Omega) \hookrightarrow Y(\Omega)
    \end{equation*}
    holds for every $\Omega \in \mathcal{J}_\alpha$. Furthermore, $X(0,1)$ is the largest r.i.~space with this property.
 \end{theorem}
The proof of Theorem \ref{T:optimal-ri-domain-mazya} is a~simple modification of the proof in the paper and therefore is omitted. Note that, by \cite[Proposition~3.3]{Mi23} the theorem may be strengthened by omitting condition \eqref{E:KPembedding}. 

We shall now discuss the fundamental Orlicz spaces of r.i.~spaces. The following theorem was established and proved by \cite{MuPiTa22}.

\begin{theorem}[the principal alternative]\label{T:principal-alternative-for-spaces}
\begin{enumerate}
\item\label{en:PA-spaces-target} Let $Y$ be an r.i.~space and $L(Y)$ its fundamental Orlicz space.
Then either $Y\subset L(Y)$ and $L(Y)$ is the smallest Orlicz space containing $Y$,
or $Y\not\subset L(Y)$ and no smallest Orlicz space containing $Y$ exists.
\item\label{en:PA-spaces-domain} Let $X$ be an r.i.~space and $L(X)$ its fundamental Orlicz space.
Then either $L(X)\subset X$ and $L(X)$ is the largest Orlicz space contained in space $X$,
or $L(X)\not\subset X$ and no largest Orlicz space contained in $X$ exists.
\end{enumerate}
\end{theorem}

Next, we specify the principal alternative to Sobolev embeddings. The following theorem is an adjustment to \citet[Theorem 4.1]{MuPiTa22}, its proof is a simple modification of the proof in the original paper and therefore is omitted.

\begin{theorem}[principal alternative for Sobolev embeddings] \label{T:principal-alternative-sobolev-embeddings}
    Let $m$, $n \in \mathbb{N}$, $\alpha \in [\tfrac{1}{n'},1)$, and let $\| \cdot \|_{X(0,1)}$, $\| \cdot \|_{Y(0,1)}$ be r.i.~norms.
    \begin{enumerate}
        \item \label{en:PA-sobolev-embeddings-domain} If $X$ is the largest among all r.i.~spaces rendering embedding \eqref{E:sobolev} true for every $\Omega \in \mathcal{J}_\alpha$, then either $L(X)(0,1) \subset X(0,1)$ and $L^A = L(X)$ is the largest Orlicz space such that
        \begin{equation*}
            W^{m}L^A(\Omega) \hookrightarrow Y(\Omega)
        \end{equation*}
        holds for every $\Omega \in \mathcal{J}_\alpha$, or no such largest Orlicz space exists.
        \item \label{en:PA-sobolev-embeddings-target} If $Y$ is the smallest among all r.i.~spaces rendering embedding \eqref{E:sobolev} true for every $\Omega \in \mathcal{J}_\alpha$, then either $Y(0,1) \subset L(Y)(0,1)$ and $L^B = L(Y)$ is the smallest Orlicz space such that
        \begin{equation*}
            W^{m}X(\Omega) \hookrightarrow L^B(\Omega)
        \end{equation*}
        holds for every $\Omega \in \mathcal{J}_\alpha$, or no such smallest Orlicz space exists.
    \end{enumerate}
\end{theorem}

\chapter{Main results}

First, we shall present a general result concerning the fundamental function of an operator-induced space.

\begin{theorem}[fundamental function of an operator-induced space]\label{T:fundamental-optimal-domain-hardy} \ \\
Let $I\colon[0,1] \rightarrow [0,\infty)$ be a non-decreasing measurable function such that the function $\left( (0,1) \ni t \mapsto \int_t^1 \frac{1}{I(s)} \,\mathrm{d}s \, \right)$ belongs to $Y(0,1)$, where $Y(0,1)$ is an r.i.~space such that
\begin{equation}\label{E:fundamental-optimal-domain-hardy-pres}
    \lim_{t \to 0^+}{\varphi_Y(t)} = 0,
\end{equation}
and
\begin{equation}\label{E:fundamental-optimal-domain-hardy-presup}
    \int_{0}^{t} \sup_{\tau \in (s,1)} \frac{\varphi_{Y}(\tau)}{I(\tau)} \, \mathrm{d}s \lesssim t \sup_{s \in (t,1)} \frac{\varphi_{Y}(\frac{s}{2})}{s} \int_{\frac{s}{2}}^s \frac{\mathrm{d} \tau}{I(\tau)} , \quad t \in (0,1).
\end{equation}
Let $X(0,1)$ be defined by
\begin{equation*}
    \left\Vert f \right\Vert_{X(0,1)} = \sup_{h} \left\Vert \int_{\tau}^1 \frac{h(s)}{I(s)} \mathrm{d}s \right\Vert_{Y(0,1)} \quad \mathrm{for}~ f \in \mathcal{M}(0,1),
\end{equation*}
where the supremum is extended over all $h \in \mathcal{M}_+(0,1)$ such that $h^* = f^*$. Then $X(0,1)$ is an r.i.~space, and one has
\begin{equation*}
    \varphi_X(t) \approx t \sup_{s \in (t,1)} \frac{\varphi_Y(\frac{s}{2})}{s} \int_{\frac{s}{2}}^s \frac{\mathrm{d}\tau}{I(\tau)} \quad \mathrm{for}~t \in (0,1).
\end{equation*}

\end{theorem}

\begin{proof}
    The fact that $X$ is an r.i.~space follows from Theorem \ref{T:optimal-ri-domain-mazya}. First, we will examine the lower bound of $\varphi_X$. Fix $t \in (0,1)$. By the boundedness of the dilation operator on $Y$, we get
    \begin{equation*}
    \begin{split}
        \varphi_X(t) & = \| \chi_{[0,t)} \|_{X(0,1)} \geq \left\Vert \int_{\tau}^1 \frac{\chi_{(0,t)}(s)}{I(s)} \mathrm{d}s \right\Vert_{Y(0,1)} \geq \left\Vert \chi_{(0,\frac{t}{2})}(\tau) \int_{\tau}^t \frac{\mathrm{d}s}{I(s)} \right\Vert_{Y(0,1)} \\
        & \geq \left\Vert \chi_{(0,\frac{t}{2})}(\tau) \int_{\frac{t}{2}}^t \frac{\mathrm{d}s}{I(s)} \right\Vert_{Y(0,1)} = \varphi_Y(\tfrac{t}{2}) \int_{\frac{t}{2}}^t \frac{\mathrm{d}s}{I(s)},
    \end{split}
    \end{equation*}
    showing that
    \begin{equation*}
        \frac{\varphi_X(t)}{t} \geq \frac{\varphi_Y(\tfrac{t}{2})}{t} \int_{\frac{t}{2}}^t \frac{\mathrm{d}s}{I(s)} \quad \mathrm{for~every}~t \in (0,1).
    \end{equation*}
    Since the function $t \mapsto \frac{\varphi_X(t)}{t}$ is non-increasing, it follows that
    \begin{equation*}
        \frac{\varphi_X(t)}{t} \geq \sup_{s \in (t,1)} \frac{\varphi_Y(\tfrac{s}{2})}{s} \int_{\frac{s}{2}}^s \frac{\mathrm{d}\tau}{I(\tau)} \quad \mathrm{for~every}~t\in (0,1).
    \end{equation*}
    Thus, we have obtained the lower bound of $\varphi_X$, as
    \begin{equation}\label{E:fundamental-optimal-domain-hardy-LB}
            \varphi_X(t) \geq t \sup_{s \in (t,1)} \frac{\varphi_Y(\tfrac{s}{2})}{s} \int_{\frac{s}{2}}^s \frac{\mathrm{d}\tau}{I(\tau)} \quad \mathrm{for~every}~t\in (0,1).
    \end{equation}
    Let us focus on the upper bound of $\varphi_X$. Note that since $\varphi_Y$ is locally absolutely continuous on $(0,1]$, it is differentiable a.e.~and for any measurable $f$, one has
    \begin{equation*}
        \|f\|_{\Lambda(Y)} = \varphi_Y(0_+)\|f\|_{\infty} + \int_{0}^{1} f^* \varphi'_Y.
    \end{equation*}
    Therefore, by the fundamental embedding \eqref{E:Lorentz-Marcinkiewicz-Sandwich},~\eqref{E:fundamental-optimal-domain-hardy-pres} and Fubini's theorem, we obtain that for any $h \geq 0$, it holds that
    \begin{equation*}
    \begin{split}
        \left\Vert \int_{\tau}^1 \frac{h(s)}{I(s)} \mathrm{d}s \right\Vert_{Y(0,1)} & \lesssim \left\Vert \int_{\tau}^1 \frac{h(s)}{I(s)} \mathrm{d}s \right\Vert_{\Lambda (Y(0,1))} = \int_0^1 \int _{\tau}^1 \frac{h(s)}{I(s)} \, \mathrm{d}s \, \varphi'_Y(\tau) \, \mathrm{d}\tau \\
        & = \int_0^1 \frac{h(s)}{I(s)} \int_0^s \varphi'_Y(\tau) \, \mathrm{d}\tau \, \mathrm{d}s = \int_0^1 \frac{h(s)}{I(s)} \varphi_Y(s) \, \mathrm{d}s \\
        & \leq \int_0^1 h(s) \sup_{\tau \in (s,1)} \frac{\varphi_Y(\tau)}{I(\tau)} \, \mathrm{d}s \,. 
    \end{split}
    \end{equation*}
    Hence, by the definition of $\varphi_X$ and by the Hardy-Littlewood inequality \eqref{H-L}, it holds for any $t \in (0,1)$, that
    \begin{equation*}
    \begin{split}
        \varphi_X(t) & = \| \chi_{[0,t)} \|_{X(0,1)} = \sup_{h^*=\chi_{[0,t)}} \left\Vert \int_{\tau}^1 \frac{h(s)}{I(s)} \, \mathrm{d}s \right\Vert_{Y(0,1)} \\
        & \leq \sup_{h^*=\chi_{[0,t)}} \int_0^1 h(s) \sup_{\tau \in (s,1)} \frac{\varphi_Y(\tau)}{I(\tau)} \, \mathrm{d}s \\
        & \leq \sup_{h^*=\chi_{[0,t)}} \int_0^1 h^*(s) \sup_{\tau \in (s,1)} \frac{\varphi_Y(\tau)}{I(\tau)} \, \mathrm{d}s \\
        & = \int_0^t \sup_{\tau \in (s,1)} \frac{\varphi_Y(\tau)}{I(\tau)} \, \mathrm{d}s \, .
    \end{split}
    \end{equation*}
    These inequalities give us an upper bound of $\varphi_X$, as
    \begin{equation}\label{E:fundamental-optimal-domain-hardy-UB}
        \varphi_X(t) \leq \int_0^t \sup_{\tau \in (s,1)} \frac{\varphi_Y(\tau)}{I(\tau)} \, \mathrm{d}s \quad \mathrm{for~every}~t \in (0,1).
    \end{equation}
    Thus, by combining the lower bound \eqref{E:fundamental-optimal-domain-hardy-LB} and the upper bound \eqref{E:fundamental-optimal-domain-hardy-UB} of $\varphi_X$ for $t \in (0,1)$, we obtain the estimates
    \begin{equation*} 
            t \sup_{s \in (t,1)} \frac{\varphi_Y(\tfrac{s}{2})}{s} \int_{\frac{s}{2}}^s \frac{\mathrm{d}\tau}{I(\tau)} \leq \varphi_X(t) \lesssim \int_0^t \sup_{\tau \in (s,1)} \frac{\varphi_Y(\tau)}{I(\tau)} \, \mathrm{d}s \,.
    \end{equation*}
    Finally, applying inequality \eqref{E:fundamental-optimal-domain-hardy-presup} to the estimates above, we get the desired result
    \begin{equation*}
        \varphi_X(t) \approx t \sup_{s \in (t,1)} \frac{\varphi_Y(\frac{s}{2})}{s} \int_{\frac{s}{2}}^s \frac{\mathrm{d}\tau}{I(\tau)} \quad \mathrm{for}~t \in (0,1).
    \end{equation*}
\end{proof}
The particular case where $\Omega$ is a Lipschitz domain was treated in \cite[Theorem~4.2]{MuPiTa22}.

We shall now use the result obtained in Theorem \ref{T:fundamental-optimal-domain-hardy} to extract important information about fundamental functions of optimal domain r.i.~spaces in Sobolev embeddings corresponding to given target spaces sharing the same fundamental level.

\begin{corollary} \label{T:fundamental-optimal-domain-corollary}
    Let $\Omega \subset \mathbb{R}^n$ be with isoperimetric profile $I_{\Omega}$. Suppose that $Y_1$, $Y_2$ are r.i.~spaces over $\Omega$ on the same fundamental level, i.e. $\varphi_{Y_1} \approx \varphi_{Y_2}$, and \eqref{E:fundamental-optimal-domain-hardy-pres} and \eqref{E:fundamental-optimal-domain-hardy-presup} hold for both $Y_1$ and $Y_2$ in place of $Y$. Let $X_j$ be the optimal r.i.~domain spaces in the embedding
    \begin{equation*}
        W^m X_j(\Omega) \hookrightarrow Y_j(\Omega),
    \end{equation*}
    for $j = 1,2$. Then $X_1$, $X_2$ are also on the same fundamental level, i.e. $\varphi_{X_1} \approx \varphi_{X_2}$.
\end{corollary}
\begin{proof}
    The proof follows immediately from the formula of the norm in the optimal r.i.~domain space \eqref{E:optimal-ri-domain} and by Theorem \ref{T:fundamental-optimal-domain-hardy}.
\end{proof}

The presuppositions of Theorem~\ref{T:fundamental-optimal-domain-hardy}, namely \eqref{E:fundamental-optimal-domain-hardy-presup}, aren't particularly flexible, as there possibly exist classes of isoperimetric functions for which the presupposition is needlessly strong, or even entirely unnecessary. We aim to prove that for Maz'ya classes of domains, Theorem \ref{T:fundamental-optimal-domain-hardy} may be applied to their isoperimetric function directly, omitting the need for presupposition \eqref{E:fundamental-optimal-domain-hardy-presup} entirely. To prove this, we must find a work-around, so we do not need to consider every domain in $\mathcal{J}_\alpha$.

In order to obtain the necessary condition for embedding \eqref{E:sobolev} for some $\alpha \in [\tfrac{1}{n'},1)$ in the reduction principle (Theorem~\ref{T:reduction-principle-mazya}), we need that the embedding holds for every $\Omega \in \mathcal{J}_\alpha$. We will however point out that it suffices to focus on the worst domain with such property, denoted $\Omega_\alpha$. To describe this domain, we use \cite[Proposition 10.1 (i)]{CiPiSl15}, which follows. In the statement, $\omega_{n-1}$ denotes the Lebesgue measure of the unit ball in $\mathbb{R}^{n-1}$.
\begin{proposition} \label{T: worst-domain}
    Let $n \in \mathbb{N}$, $n \geq 2$, $\alpha \in [\tfrac{1}{n'},1)$. Define $\eta_\alpha\colon [0, \tfrac{1}{1-\alpha}] \rightarrow [0,\infty)$ as
    \begin{equation*}
        \eta_\alpha(t) = \omega_{n-1}^{-\frac{1}{n-1}} \, (1 - (1-\alpha)t)^{\frac{\alpha}{(1-\alpha)(n-1)}} \quad \mathrm{for~} t \in [0, \tfrac{1}{1-\alpha}].
    \end{equation*}
    Let $\Omega_\alpha$ be the Euclidean domain in $\mathbb{R}^n$ given by
    \begin{equation*}
        \Omega_\alpha = \left\{ (x, x_n) \in \mathbb{R}^n\colon x \in \mathbb{R}^{n-1},~0 < x_n < \tfrac{1}{1-\alpha},~|x| < \eta_\alpha(x_n) \right\}.
    \end{equation*}
    Then $|\Omega_\alpha| = 1$, and
    \begin{equation}\label{E:worst-domain}
        I_{\Omega_\alpha}(t) \approx t^\alpha \quad \mathrm{for~} t \in [0, \tfrac{1}{2}].
    \end{equation}
\end{proposition}

The best case scenario happens when $\alpha = \frac{1}{n'}$ and we are dealing with John domains. For example, if $n = 2$ and $\alpha = \frac{1}{2}$, then the function $\eta_\alpha$ takes the form
\begin{equation*}
    \eta_\frac{1}{2}(t) = \tfrac{1}{2}(1-\tfrac{1}{2}t)
\end{equation*}
and, as a result,
\begin{equation*}
    \Omega_\alpha = \{ x_1 \in (0,2),~|x_2| < \tfrac{1}{2}(1-\tfrac{1}{2}x_1) \}.
\end{equation*}
In other words, $\Omega_\alpha$ takes the form of a triangle with vertices (denoted by $[x_1, x_2]$) $[0,\tfrac{1}{2}]$, $[0,-\tfrac{1}{2}]$, and $[2,0]$.  We may easily verify that $|\Omega_\alpha| = 1$. The domains start to get worse as we choose $\alpha > \tfrac{1}{2}$. It is on these domains $\Omega_\alpha$ that we consider the following corollary.

\begin{corollary}\label{C:no-need-for-condition}
    Let $n \in \mathbb{N}$, $n \geq 2$, $\alpha \in [\tfrac{1}{n'},1)$. Define $\Omega_\alpha$ as in Proposition \ref{T: worst-domain}. Then, for $I_{\Omega_\alpha}(t)$, Theorem \ref{T:fundamental-optimal-domain-hardy} and Corollary \ref{T:fundamental-optimal-domain-corollary} are applicable without the restriction \eqref{E:fundamental-optimal-domain-hardy-presup}.
\end{corollary}

\begin{proof}
    Let $\alpha \in [\tfrac{1}{n'},1)$. Returning to the proof of Theorem \ref{T:fundamental-optimal-domain-hardy}, in the case of $I_{\Omega_\alpha}$, the lower bound \eqref{E:fundamental-optimal-domain-hardy-LB} contains a computable integral. Thus, it holds that for some constant $c>0$
    \begin{equation*}
    \begin{split}
        \varphi_X(t) & \gtrsim t \sup_{s \in (t,1)} \frac{\varphi_Y(\tfrac{s}{2})}{s} \int_{\frac{s}{2}}^s \frac{\mathrm{d}\tau}{\tau^\alpha} = t \sup_{s \in (t,1)} \frac{\varphi_Y(\tfrac{s}{2})}{s} \, \frac{1}{1-\alpha} \, \big( s^{1-\alpha} - (\tfrac{s}{2})^{1-\alpha}\big) \\
        & = t \sup_{s \in (t,1)} \varphi_Y(\tfrac{s}{2})s^{-\alpha}c \,.
    \end{split}
    \end{equation*}
    Therefore,
    \begin{equation*}
        \varphi_X(t) \gtrsim t \sup_{s \in (t,1)} \varphi_Y(\tfrac{s}{2})s^{-\alpha}.
    \end{equation*}
    We observe that, owing to our choice of $\Omega_\alpha$ and~\eqref{E:worst-domain}, this computation cannot produce a smaller lower bound for any $\Omega \in \mathcal{J}_\alpha$. Moreover, it follows from \eqref{E:fundamental-optimal-domain-hardy-UB} that
    \begin{equation*}
        \varphi_X(t) \lesssim \int_0^t \sup_{\tau \in (s,1)} \frac{\varphi_Y(\tau)}{\tau^\alpha} \, \mathrm{d}s \,.
    \end{equation*}
    As such, to prove the corollary, it is sufficient to prove that for any $t \in (0,1)$
    \begin{equation}\label{E:fundamental-optimal-domain-hardy-mazya-wish}
        t \sup_{s \in (t,1)} \varphi_Y(\tfrac{s}{2})s^{-\alpha} \gtrsim \int_0^t \sup_{\tau \in (s,1)} \frac{\varphi_Y(\tau)}{\tau^\alpha} \, \mathrm{d}s \,.
    \end{equation}
    Define
    \begin{equation*}
        \psi(t) := t \sup_{s \in (t,1)} \varphi_Y(\tfrac{s}{2})s^{-\alpha} \quad t \in (0,1).
    \end{equation*}
    We claim that
    \begin{equation}\label{E:fundamental-optimal-domain-hardy-mazya-homo}
        \psi(t) = t \sup_{s \in (t,1)} \varphi_Y(\tfrac{s}{2})s^{-\alpha} = \sup_{s \in (0,1)} \varphi_Y(\tfrac{s}{2}) \min \big\{t^{1-\alpha},ts^{-\alpha}\big\}.
    \end{equation}
    Indeed,
    \begin{equation*}
    \begin{split}
        \psi(t) & = t \sup_{s \in (t,1)} \varphi_Y(\tfrac{s}{2})s^{-\alpha} = \max \big\{ \varphi_Y(\tfrac{t}{2}) \, t^{1-\alpha} \, ; t \sup_{s \in (t,1)} \varphi_Y(\tfrac{s}{2})s^{-\alpha} \big\} \\
        & = \max \big\{ \sup_{s \in (0,t]} \varphi_Y(\tfrac{s}{2})t^{1-\alpha} \, ; t \sup_{s \in (t,1)} \varphi_Y(\tfrac{s}{2})s^{-\alpha} \big\} \\
        & = \max \big\{ \sup_{s \in (0,t]} \varphi_Y(\tfrac{s}{2}) \min \big\{t^{1-\alpha},ts^{-\alpha}\big\}; \sup_{s \in (t,1)} \varphi_Y(\tfrac{s}{2}) \min \big\{t^{1-\alpha},ts^{-\alpha}\big\} \big\} \\
        & = \sup_{s \in (0,1)} \varphi_Y(\tfrac{s}{2}) \min \big\{t^{1-\alpha},ts^{-\alpha}\big\}.
    \end{split}
    \end{equation*}
    Therefore, $\psi$ is non-decreasing. It also follows from equality \eqref{E:fundamental-optimal-domain-hardy-mazya-homo} that for any $\sigma \in (0,1)$ and any $k \in \mathbb{N}$, one has $ \psi(t\sigma^k) \leq \sigma^{(1-\alpha) k} \psi(t)$ for $t \in (0,1)$ and therefore
    \begin{equation*}
    \begin{split}
        \int_0^t \frac{\psi(s)}{s} \, \mathrm{d}s & = \sum_{k=0}^\infty \int_{t\sigma^{k+1}}^{t\sigma^k} \frac{\psi(s)}{s} \, \mathrm{d}s \leq \sum_{k=0}^\infty \psi (t\sigma^k) \int_{t\sigma^{k+1}}^{t\sigma^k} \frac{\mathrm{d}s}{s} \\
        & \leq \psi(t) \log \frac{1}{\sigma} \sum_{k=0}^\infty \sigma^{(1-\alpha) k} = \psi (t) \frac{\log \tfrac{1}{\sigma}}{1-\sigma^{(1-\alpha)}}.
    \end{split}
    \end{equation*}
    Hence,
    \begin{equation*}
        \psi(t) \gtrsim \int_0^t \frac{\psi(s)}{s} \mathrm{d}s,
    \end{equation*}
    and by definition of $\psi$, we obtain the desired result \eqref{E:fundamental-optimal-domain-hardy-mazya-wish}.
\end{proof}

Combining all the results of this chapter so far yields the following theorem, which is an extension of \cite[Theorem 4.9]{MuPiTa22}.

\begin{theorem}[nonexistence of an optimal Orlicz space] \label{T:nonexistence-optimal-orlicz-on-level}
    Let $m,n \in \mathbb{N}$, let $\alpha \in [\tfrac{1}{n'},1)$, and let $Y(\Omega)$ be an r.i.~space. Assume there is no largest Orlicz space $L^A$ such that
    \begin{equation}\label{E:orlicz-to-marcinkiewicz}
        W^mL^A(\Omega) \hookrightarrow M(Y)(\Omega)
    \end{equation}
    for every $\Omega$ in $\mathcal{J_\alpha}$. Then, there is no largest Orlicz space $L^A$ such that
    \begin{equation}\label{E:orlicz-to-y}
        W^mL^A(\Omega) \hookrightarrow Y(\Omega)
    \end{equation}
    for every $\Omega$ in $\mathcal{J_\alpha}$.
\end{theorem}

\begin{proof}
    We know that $Y(\Omega)$ and $M(Y)(\Omega)$ share the same fundamental function. We denote by $X_Y(\Omega)$ and $X_{M(Y)}(\Omega)$ the largest domain r.i.~spaces in embeddings
    \begin{equation*}
        W^mX_{Y}(\Omega) \hookrightarrow Y(\Omega)
    \end{equation*}
    and
    \begin{equation*}
        W^mX_{M(Y)}(\Omega) \hookrightarrow M(Y)(\Omega),
    \end{equation*}
    respectively. Then, by Corollary~\ref{T:fundamental-optimal-domain-corollary} combined with Corollary~\ref{C:no-need-for-condition}, we obtain that $\varphi_{X_Y}=\varphi_{X_{M(Y)}}$. Owing to the assumption, there is no largest Orlicz space $L^A(\Omega)$ in the Sobolev 
    embedding~\eqref{E:orlicz-to-marcinkiewicz}. Therefore, by Theorem~\ref{T:principal-alternative-sobolev-embeddings}, it follows that $L_{X_{M(Y)}}(\Omega)\not\hookrightarrow X_{M(Y)}(\Omega)$. Since it holds that $Y(\Omega)\hookrightarrow 
    M(Y)(\Omega)$, it is also true that $X_Y(\Omega)\hookrightarrow X_{M(Y)}(\Omega)$. Consequently, $L_{X_{M(Y)}}(\Omega)\not\hookrightarrow X_Y(\Omega)$. But, since $\varphi_{X_Y}=\varphi_{X_{M(Y)}}$, we have 
    $L_{X_{M(Y)}}(\Omega)=L_{X_Y}(\Omega)$. Altogether, $L_{X_{Y}}(\Omega)\not\hookrightarrow X_Y(\Omega)$, whence, using Theorem~\ref{T:principal-alternative-sobolev-embeddings} once again, there is no largest Orlicz space 
    $L^A(\Omega)$ in the Sobolev embedding~\eqref{E:orlicz-to-y}.
\end{proof}

We shall now apply Theorem \ref{T:nonexistence-optimal-orlicz-on-level} to show that there is no optimal Orlicz domain space $L^A(\Omega)$ in the embedding
\begin{equation*}
    W^mL^A(\Omega) \hookrightarrow L^{\infty,q,-1+(1-\alpha)m-\frac{1}{q}}(\Omega)\quad\text{for $q \in [\tfrac{1}{1-m(1-\alpha)}, \infty]$,}
\end{equation*}
thereby solving the open problem mentioned in the introduction. The value $q=\tfrac{1}{1-m(1-\alpha)}$ corresponds to the case in which the target space is optimal in the embedding of the limiting Sobolev space, see~\eqref{E:general-bw}.

First, we shall prove two lemmas, which we will later use in the proof. The first lemma concerns the r.i.~norm of the characteristic function.

\begin{lemma} \label{T: ri-norm-characteristic-function}
    Let $\zeta > 0$, $a \in (0,1)$ and $\| \cdot \|_{X(0,1)}$ be an r.i.~norm. Then
    \begin{equation*}
        \| \chi_{(0,a)}(t)(a^\zeta - t^\zeta) \|_{X(0,1)} \approx a^\zeta \| \chi_{(0,a)}\|_{X(0,1)}.
    \end{equation*}
\end{lemma}
\begin{proof}
    The inequality
    \begin{equation*}
        \| \chi_{(0,a)}(t)(a^\zeta - t^\zeta) \|_{X(0,1)} \lesssim a^\zeta \| \chi_{(0,a)} \|_{X(0,1)}   
    \end{equation*}
    is obvious. Conversely, we observe that
    \begin{equation*}
    \begin{split}
        \| \chi_{(0,a)}(t)(a^\zeta - t^\zeta) \|_{X(0,1)} & \geq \| \chi_{(0,\frac{a}{2})}(t)(a^\zeta - t^\zeta) \|_{X(0,1)} \\
        & \geq \| \chi_{(0,\frac{a}{2})}(t)(a^\zeta - (\tfrac{a}{2})^\zeta) \|_{X(0,1)} \\
        & = a^\zeta (1-(\tfrac{1}{2})^\zeta) \| \chi_{(0,\frac{a}{2})}  \|_{X(0,1)} \\
        & \gtrsim a^\zeta \| \chi_{(0,a)}\|_{X(0,1)},
    \end{split}
    \end{equation*}
    since, by the triangle inequality and by the rearrangement-invariance of $\| \cdot \|_{X(0,1)}$,
    \begin{equation*}
        \| \chi_{(0,a)} \|_{X(0,1)} = \| \chi_{(0,\frac{a}{2}]} + \chi_{(\frac{a}{2},a)}\|_{X(0,1)} \leq \| \chi_{(0,\frac{a}{2})} \|_{X(0,1)} + \| \chi_{(\frac{a}{2},a)} \|_{X(0,1)}
    \end{equation*}
    \begin{equation*}
        \, \, \leq 2 \| \chi_{(0,\frac{a}{2})} \|_{X(0,1)}. \qquad \qquad \qquad \qquad \qquad \quad
    \end{equation*}
    Altogether, one has
    \begin{equation*}
        \| \chi_{(0,a)}(t)(a^\zeta - t^\zeta) \|_{X(0,1)} \approx a^\zeta \| \chi_{(0,a)}\|_{X(0,1)},
    \end{equation*}
    with the constants of equivalence depending only on $\zeta$.
\end{proof}

The second lemma concerns the boundedness of the operator $T_\gamma$ introduced in~\eqref{E:sup-operator}.

\begin{lemma} \label{T: supremum-operator}
    Let $\gamma \in (0,1)$. Then
    \begin{equation}\label{E:t-gamma}
        T_\gamma\colon L^{1,1;1-\gamma}(0,1) \rightarrow L^{1,1;1-\gamma}(0,1).
    \end{equation}
\end{lemma}
\begin{proof}
    This is a particular case of \cite[Theorem 3.2]{GoOpPi06}, applied to $p=q=1$, and $u(t) = t^\gamma$, $v(t) = w(t) = (\log\tfrac{2}{t})^{1-\gamma}$ for $t \in (0,1)$. It is readily verified that the necessary and sufficient condition for~\eqref{E:t-gamma}, namely 
    \begin{equation*}
        \int_{0}^t s^{-\gamma}\sup_{s\le \tau\le t}u(\tau) w(s)\,ds
        \le C
        \int_{0}^{t}v(s)\,ds
        \quad\text{for every $t\in(0,1)$,}
    \end{equation*}
    holds.
\end{proof}

We are now in position to formulate and prove our main result.

\begin{theorem} \label{T:nonexistence-optimal-orlicz-sobolev}
    Let $n \in \mathbb{N}$, $n \geq 2$, $\alpha \in [\tfrac{1}{n'},1)$, $m \in \mathbb{N}$ such that $m < \tfrac{1}{1-\alpha}$, and let $q \in [\tfrac{1}{m(1-\alpha)}, \infty]$. Then there is no largest Orlicz space $L^A$ such that
    \begin{equation}\label{E:embedding-BW}
        W^mL^A(\Omega) \hookrightarrow L^{\infty,q,-1+(1-\alpha)m-\frac{1}{q}}(\Omega)
    \end{equation}
    for every $\Omega$ in $\mathcal{J}_\alpha$.
\end{theorem}

\begin{proof}
    By \cite[Lemma 3.7]{OpPi99}, for every $q \in [\tfrac{1}{1-m(1-\alpha)}, \infty]$, the spaces $L^{\infty,q,-1+(1-\alpha)m-\frac{1}{q}}$ and $\mathrm{exp} \,L^\frac{1}{1-m(1-\alpha)}$ share the same fundamental function $\varphi$, where
    \begin{equation*}
        \varphi(t) \approx \frac{1}{{(\log \frac{2}{t})}^{1-m(1-\alpha)}} \quad \mbox{for}~t \in (0,1),
    \end{equation*}
    and it holds that $L^{\infty,q,-1+(1-\alpha)m-\frac{1}{q}} \hookrightarrow \mathrm{exp} \,L^\frac{1}{1-m(1-\alpha)}$. Moreover, by the aforementioned lemma, it holds that $M(Z) = \mathrm{exp} \,L^\frac{1}{1-m(1-\alpha)}$ for any r.i.~space $Z$ with the fundamental function $\varphi$, and in fact, $L(Z) = \mathrm{exp} \,L^\frac{1}{1-m(1-\alpha)}$.

    We aim to prove that there is no largest Orlicz space $L^A$ such that the embedding \eqref{E:embedding-BW} holds for every $\Omega$ in $\mathcal{J}_\alpha$. Therefore, by Theorem \ref{T:nonexistence-optimal-orlicz-on-level}, it suffices to prove that there is no largest Orlicz space $L^A$ such that the embedding
    \begin{equation}\label{E:embedding-T}
        W^mL^A(\Omega) \hookrightarrow \mathrm{exp} \,L^\frac{1}{1-m(1-\alpha)}(\Omega)
    \end{equation}
    holds for every $\Omega$ in $\mathcal{J}_\alpha$. By Theorem \ref{T:principal-alternative-sobolev-embeddings} \ref{en:PA-sobolev-embeddings-domain}, it is therefore enough to prove that if $X$ is the optimal domain r.i.~space in
    \begin{equation*}
        W^mX(\Omega) \hookrightarrow \mathrm{exp} \,L^\frac{1}{1-m(1-\alpha)}(\Omega)
    \end{equation*}
    for every $\Omega$ in $\mathcal{J}_\alpha$, then $L(X)(0,1) \not\subset X(0,1)$.

    First, we find the optimal domain space $X(\Omega)$. We use Theorem \ref{T:optimal-ri-domain-mazya} and obtain the formula
    \begin{equation*}
        \|f\|_{X(0,1)} = \sup_{h \sim f} \left\Vert \int_t^1 s^{-1+m(1-\alpha)}h(s) \, \mathrm{d}s \right\Vert_{\mathrm{exp} \,L^\frac{1}{1-m(1-\alpha)}(0,1)} \quad \mbox{for}~f \in \mathcal{M}(0,1).
    \end{equation*}
    By Lemma \ref{T: supremum-operator}, applied to $\gamma = m(1-\alpha)$, the operator $T_{m(1-\alpha)}$ is bounded on $L^{1,1;1-m(1-\alpha)}$. Furthermore, by \cite[Theorem~6.11]{OpPi99} and \eqref{Orlicz-exp}, it holds that $(L^{1,1;1-m(1-\alpha)})'(0,1) = \exp L^{\frac{1}{1-m(1-\alpha)}}(0,1)$. Hence, it follows from \cite[Theorem 4.7]{EdMiMuPi19} that
    \begin{equation*}
        \|f\|_{X(0,1)} \approx  \left\Vert \int_t^1 s^{-1+m(1-\alpha)}f^*(s) \, \mathrm{d}s \right\Vert_{\mathrm{exp} \,L^\frac{1}{1-m(1-\alpha)}(0,1)}.
    \end{equation*}
    As such, we will with no loss of generality suppose that 
    \begin{equation*}
        \|f\|_{X(0,1)} = \left\Vert \int_t^1 s^{-1+m(1-\alpha)}f^*(s) \, \mathrm{d}s \right\Vert_{\mathrm{exp} \,L^\frac{1}{1-m(1-\alpha)}(0,1)}.
    \end{equation*}
    
    Next, we describe $L(X)$. For that purpose, we need to find $\varphi_X$ so we can determine its Young function $A$. Let $a \in (0,1)$. Then, as $\chi_{(0,a)}$ = $\chi_{(0,a)}^*$ and by Lemma \ref{T: ri-norm-characteristic-function}, we get
    \begin{equation*}
    \begin{split}
        \varphi_X(a) & = \|\chi_{(0,a)}\|_{X(0,1)} = \left\Vert \int_t^1 s^{-1+m(1-\alpha)}\chi_{(0,a)}(s) \, \mathrm{d}s \right\Vert_{\mathrm{exp} \,L^\frac{1}{1-m(1-\alpha)}(0,1)} \\
        & = \left\Vert \chi_{(0,a)}(t) \int_t^a s^{-1+m(1-\alpha)} \mathrm{d}s \right\Vert_{\mathrm{exp} \,L^\frac{1}{1-m(1-\alpha)}(0,1)} \\
        & \approx \left\Vert \chi_{(0,a)}(t) (a^{m(1-\alpha)} - t^{m(1-\alpha)}) \right\Vert_{\mathrm{exp} \,L^\frac{1}{1-m(1-\alpha)}(0,1)} \\
        & \approx \left\Vert \chi_{(0,a)} \right\Vert_{\mathrm{exp} \,L^\frac{1}{1-m(1-\alpha)}(0,1)} \cdot a^{m(1-\alpha)} = \frac{1}{(\log \frac{2}{a})^{1-m(1-\alpha)}} \cdot a^{m(1-\alpha)}.
    \end{split}
    \end{equation*}
    We therefore know that
        \begin{equation*}
            \varphi_X(a) \approx \frac{a^{m(1-\alpha)}}{(\log \frac{2}{a})^{1-m(1-\alpha)}} \quad \mbox{for}~a \in (0,1).
        \end{equation*}
Hence, by \eqref{Orlicz-fundamental_formula} and simple computation, it holds that 
\begin{equation*}
    L(X) = L^\frac{1}{m(1-\alpha)} \, \log L^{1-\frac{1}{m(1-\alpha)}}.
\end{equation*}

    Finally, we need to prove that $L(X) \not\subset X$. We know that $X$ is optimal in the embedding $W^mX \hookrightarrow \mathrm{exp} \, L^\frac{1}{1-m(1-\alpha)}$, and consequently, it suffices to prove that $W^mL(X) \not\hookrightarrow \mathrm{exp}\,L^\frac{1}{1-m(1-\alpha)}$. Then, by Theorem \ref{T:reduction-principle-mazya}, it is enough to prove that the inequality
    \begin{equation*}
        \left\Vert \int_t^1 s^{-1+m(1-\alpha)}f(s) \, \mathrm{d}s \right\Vert_{\mathrm{exp} \,L^\frac{1}{1-m(1-\alpha)}(0,1)} \lesssim \|f\|_{L(X)(0,1)} 
    \end{equation*}
    does not hold.
    Let $f(t) = t^{-m(1-\alpha)} (\log \tfrac{2}{t})^\beta$, where $\beta$ is anywhere in the interval $(-m(1-\alpha), 1-2m(1-\alpha))$. Owing to the assumption, the interval is non-empty. First, we prove that $f \in L(X)$, that is, $\|f\|_{L(X(0,1))}<\infty$. Since $A(t)$ satisfies the $\Delta_2$ condition, it holds that $f \in L^A \iff \int_0^1A(f) < \infty$. Furthermore,
    \begin{equation*}
    \begin{split}
        \int_0^1A(f) < \infty & \iff \int_0^1s^{-1}(\log \tfrac{2}{s})^{\beta\frac{1}{m(1-\alpha)}+1-\frac{1}{m(1-\alpha)}} \, \mathrm{d}s < \infty \\
        & \iff \beta < 1-2m(1-\alpha).
     \end{split}
    \end{equation*}
    Therefore, our choice of $\beta$ guarantees that $f \in L(X)$.
    Moreover, we know that 
    \begin{equation*}
        \|g\|_{\mathrm{exp} \,L^\frac{1}{1-m(1-\alpha)}(0,1)} \approx \sup_{t \in (0,1)} \frac{g^*(t)}{(\log \frac{2}{t})^{1-m(1-\alpha)}} \quad \mbox{for}~g \in \mathcal{M}(0,1),
    \end{equation*}
    thus it is sufficient to prove that
    \begin{equation}\label{E:not-in-space}
        \sup_{t \in (0,1)} \frac{\int_t^1 s^{-1+m(1-\alpha)}f(s) \, \mathrm{d}s}{(\log \frac{2}{t})^{1-m(1-\alpha)}} = \infty.
    \end{equation}
    We compute the integral and obtain
    \begin{equation*}
        \sup_{t \in (0,1)} \frac{\int_t^1 s^{-1+m(1-\alpha)}\cdot s^{-m(1-\alpha)}(\log \frac{2}{s})^\beta \, \mathrm{d}s}{(\log \frac{2}{t})^{1-m(1-\alpha)}} = \sup_{t \in (0,1)} \frac{\int_t^1 \frac{(\log \frac{2}{s})^\beta}{s}\, \mathrm{d}s}{(\log \frac{2}{t})^{1-m(1-\alpha)}}
    \end{equation*}
    \begin{equation*}
        \approx \sup_{t \in (0,1)} \frac{(\log \frac{2}{t})^{\beta + 1}}{(\log \frac{2}{t})^{1-m(1-\alpha)}} = \sup_{t \in (0,1)} (\log \tfrac{2}{t})^{\beta + m(1-\alpha)} = \infty
    \end{equation*}
    so the equality holds if $\beta > -m(1-\alpha)$. Hence, once again, our choice of $\beta$ yields~\eqref{E:not-in-space}. Therefore, we have found a function $f \in L(X)$ such that
    \begin{equation*}
        \left\Vert \int_t^1 s^{-1+m(1-\alpha)}f(s) \, \mathrm{d}s \right\Vert_{\mathrm{exp} \,L^\frac{1}{1-m(1-\alpha)}} = \infty.
    \end{equation*}
    Consequently, by Theorem \ref{T:reduction-principle-mazya} $L(X) \not\subset X$. By Theorem \ref{T:principal-alternative-sobolev-embeddings}, there is no largest Orlicz space which would render the embedding \eqref{E:embedding-T} true for every $\Omega$ in $\mathcal{J}_\alpha$. Finally, by Theorem \ref{T:nonexistence-optimal-orlicz-on-level}, there is no largest Orlicz space which would render the embedding \eqref{E:embedding-BW} true for every $\Omega$ in $\mathcal{J}_\alpha$.
\end{proof}

\begin{remark}
    In this work, we have focused primarily on Maz'ya classes of Euclidean domains. We are aware that there are plenty of open questions worth pursuing which we leave open. Pivotal examples are the case $\alpha = 1$ in Maz'ya domains, Gaussian--Sobolev embeddings, embeddings on domains endowed with Frostman-Ahlfors measures, or embeddings on probability spaces. The reason we do not consider these cases is that the corresponding integral operators get too complicated (e.g.~they take the form, at least for higher-order embeddings, of kernel-type operators). Hence such considerations reach beyond the scope of this text. We plan to study them in our following work.
\end{remark}

\bibliographystyle{plainnat}   
\renewcommand{\bibname}{Bibliography}
\bibliography{optimality-orlicz}

\openright
\end{document}